\documentclass[a4paper,11pt]{amsart}
\usepackage{hyperref}
\usepackage[utf8]{inputenc}
\usepackage[all]{xy}
\usepackage{amssymb}
\usepackage{amsmath}
\usepackage{bbm}
\usepackage[final]{showkeys}
\usepackage{leftidx}
\usepackage{yfonts}
\newtheorem{theorem}{Theorem}
\newtheorem{proposition}{Proposition}
\newtheorem{lemma}{Lemma}

\newtheorem{remark}{Remark}
\newtheorem*{remark*}{Remark}
\newtheorem{definition}{Definition}
\newtheorem*{theorem*}{Theorem}
\newtheorem*{conjecture*}{Conjecture}

\newtheorem*{notation*}{Notation}

\newtheorem*{app*}{Application to analytic GIT-quotients}
\newtheorem*{appl*}{Application}
\newtheorem*{fact}{Fact}

\numberwithin{equation}{section}

\newcommand{\C}{{\mathbb C} }

\newcommand{\cA}{{\mathcal A} }

\newcommand{\cE}{{\mathcal E} }
\newcommand{\cF}{{\mathcal F} }

\newcommand{\cI}{{\mathcal I} }

\newcommand{\cK}{{\mathcal K} }
\newcommand{\cL}{{\mathcal L} }
\newcommand{\cM}{{\mathcal M} }
\newcommand{\cN}{{\mathcal N} }
\newcommand{\cO}{{\mathcal O} }

\newcommand{\wt}{\widetilde}
\newcommand{\wh}{\widehat}

\def\ol#1{{\overline{#1}}}
\def\ul#1{{\underline{#1}}}
\def\ii{\sqrt{-1}}
\def\pt{\partial}

\def\cinf{C^\infty}
\def\we{\wedge}
\def\tr{\mathrm{tr}}

\def\s2{/\hspace{-3pt}/}
\def\san2{/\hspace{-3pt}/_{\hspace{-1pt}an}}
\def\VsG{{V\s2 G}}

\makeatletter
\def\blfootnote{\xdef\@thefnmark{}\@footnotetext}
\makeatother

\def\ka{K\"ah\-ler}
\def\he{Her\-mite-Ein\-stein}
\def\wp{Weil-Pe\-ters\-son}
\def\ks{Kodaira-Spencer}

\author{Nicholas Buchdahl}
\address{School of Mathematical Sciences\\ University of Adelaide\\ Adelaide\\
Australia 5005}
\email{nicholas.buchdahl@adelaide.edu.au}

\author{Georg Schumacher}
\address{Fachbereich Mathematik und Informatik, Philipps-Universit\"at Marburg, Lahnberge, Hans-Meerwein-Stra{\ss}e, D-35032 Marburg, Germany}
\email[Corresponding author]{schumac@mathematik.uni-marburg.de}

\begin{document}

\title{An Analytic Application of Geometric Invariant Theory}

\begin{abstract}
   Given a compact \ka\ manifold, Geometric Invariant Theory is applied to construct analytic GIT-quotients that are local models for a classifying space of (poly)stable holomorphic vector bundles containing the coarse moduli space of stable bundles as an open subspace. For local models invariant generalized \wp\ forms exist on the parameter spaces, which are restrictions of symplectic forms on smooth ambient spaces. If the underlying \ka\ manifold is of Hodge type, then the \wp\ form on the moduli space of stable vector bundles is known to be the Chern form of a certain determinant line bundle equipped with a Quillen metric. It gives rise to a holomorphic line bundle on the classifying GIT space together with a continuous hermitian metric.
\end{abstract}

\maketitle


\section{Introduction}
The aim of this article is to apply methods of Geometric Invariant Theory to Deformation Theory and construct a {\it classifying space} $\cM_{GIT}$ for (poly)stable vector bundles on compact \ka\ manifolds (Theorem~\ref{th:main}), concluding previous work from \cite{bu-sch}.  The space $\cM_{GIT}$ carries a natural complex analytic structure. As its points correspond to isomorphism classes of polystable holomorphic vector bundles, such a space is sometimes referred to as a
{\it coarse moduli space}, although it satisfies slightly weaker axioms. The notion of a {\em coarse moduli space of semistable vector bundles} in algebraic geometry, however, refers to a space whose points include certain equivalence classes of semistable bundles, where bundles with the same Narasimhan-Seshadri filtration are identified.

The underlying topological space is, by the methods of Atiyah, Hitchin, and Singer \cite[\S~6]{AHS}, and Donaldson, Kronheimer \cite[Sect.~4.2.1]{d-k} a Hausdorff space. The coarse moduli space of stable holomorphic vector bundles is contained in the classifying space as the complement of an analytic subspace.

Given a semi-universal but not universal deformation of a holomorphic vector bundle say, one cannot expect that the action of the automorphism group $G$ of the given object  on the linear space $V$ of infinitesimal deformations provides an action on the parameter space that is compatible with the respective family. However, this property holds in the case of polystable vector bundles (Theorem~\ref{th:Psi}), which was shown in \cite{bu-sch}. It is emphasized that due to Theorem~\ref{th:stab}, (also from \cite{bu-sch}) it is possible to speak about ``(poly)stable points'' referring to (poly)stable fibers or (poly)stable points  with respect to group actions so that methods of deformation theory match the GIT approach.

The analogous problem for moduli of cscK manifolds was solved by Dervan-Naumann \cite{d-n}. The role of our Theorem~\ref{th:stab} was played by the result of Sz\'ekely-hidi \cite{sz}.

The action of the reductive group $G$ on the vector space $V$ of infinitesimal deformations of a polystable bundle yields a GIT-quotient $V\s2 G$. On the other hand the parameter space $S$ for a semi-universal deformation of a polystable bundle has been realized as a closed analytic subspace of a neighborhood $U\subset V$ of the origin together with a restricted action of $G$. There exists a $G$-saturated open subset $\wt U\supset U$ and a closed, $G$-invariant, analytic subspace $\wt S\subset \wt U$ such that $S = \wt S \cap U$ (cf.\ Theorem~\ref{th:glob}).

In fact the algebraic approach to construct a GIT-quotient of $\wt S$ cannot be applied. However, the analyticity of the image of $\wt S$ in $V\s2 G$ follows from Kuhlmann's generalization of the Remmert proper mapping theorem  (Theorem~\ref{th:localmain}). The image $\wt S\san2 G$ is called analytic GIT-quotient. It is a closed analytic subset of the open set $\wt U\san2 G \subset V\s2 G$.

In case the analytic structure on $S$ is weakly normal, it follows immediately that the structure sheaf on the quotient consists of invariant holomorphic functions, or more generally the weak normalization of $\wt S\san2 G$ carries the $G$-invariant holomorphic functions on the weak normalization of $\wt S$ so that a globalization is constructed in the weakly normal category.

For the general case (cf.\ the approach by Dervan-Naumann \cite{d-n}) one observes that Heinzner's complexification theorem (\cite{he}) implies that $\wt S$ is Stein if $U$ has this property, and Roberts' theorem on $G$-sheaves \cite{r} yields that holomorphic functions on the analytic GIT-quotient are exactly the $G$-invariant functions on the pull-back to $\wt U$. In order to glue local analytic GIT-quotients together (Theorem~\ref{th:main}) there exists Snow's theorem \cite{sn}, who also proves the existence of a categorical quotient of a Stein space by a reductive group. In our situation a simple direct application of Luna's slice theorem for the action on $V$ is possible  (Theorem~\ref{th:appLuna}) by restricting an algebraic slice to the given analytic subspace. The result is the classifying space $\cM_{GIT}$ of isomorphism classes of polystable holomorphic vector bundles.

Fiber integral formulas were used in \cite{f-s} to define a generalized \wp\ form on the moduli space of compact \ka\ manifolds with constant scalar curvature, and to construct a determinant line bundle together with a Quillen metric, whose curvature is, up to a numerical constant, equal to the \wp\ form. For moduli of stable vector bundles such a construction ultimately goes back to Donaldson \cite{do}. For details see e.g. \cite{b-s}. Using secondary Bott-Chern classes and the Donaldson invariant one can find \ka\ potentials in the case of holomorphic vector bundles, whereas in the case of cscK manifolds the Mabuchi functional plays a similar role \cite{d-n}.

There are various ways to extend determinant line bundles. The general method is to extend the \wp\ form as a positive current from the complement of an analytic set into its closure, provided the given families extend as singular families of compact manifolds or as coherent sheaves respectively. This was done in \cite{sch} by means of $C^0$-estimates for Monge-Amp\`ere equations and in \cite{b-s} using Donaldson invariants respectively. Once such extensions exist, the determinant line bundles together with the Quillen metrics can be extended as singular, hermitian (holomorphic) line bundles after blowing up the boundary if necessary \cite{sch}. Such results apply to certain compactifications of the moduli spaces.

Less qualitative and more specific solutions are desirable. Dervan and Naumann succeeded in extending the determinant line bundle for analytic GIT spaces of cscK-manifolds in \cite{d-n} arriving at line bundles on these classifying spaces.

For a semi-universal deformation of a polystable vector bundle we use the fiber integral for the \wp\ form $\omega^{WP}$ (cf.\cite{b-s}) which together with Theorem~\ref{th:Psi} and \ref{th:stab} implies that it is the restriction of a symplectic form on a smooth ambient space. If $X$ is a Hodge manifold, more can be said. Techniques involving the Riemann-Roch-Hirzebruch formula by Bismut, Gillet, and Soul\'e provide a determinant line bundle together with a Quillen metric on a local parameter space such that the Chern form is equal to the \wp\ form on the stable locus. Also at polystable points the result yields the $L^2$-inner product of harmonic \ks\ tensors. The construction is compatible with descending to local analytic GIT-quotients and the gluing of local models, and it provides the classifying space with a \wp\ form that possesses a continuous $\pt\ol\pt$-potential. Then the general methods of Dervan and Naumann \cite{d-n} are applicable.

\textit{Acknowledgement.} G.Sch.\ thanks Indranil Biswas for several discussions.

\section{Analytic GIT-quotient}\label{se:anGIT0}
For the convenience of the reader we recall notions and facts from geometric invariant theory as far as these will be used here. The primary references include the work of Mumford-Fogarty-Kirwan, Ness, Hoskins, Kirwan  \cite{MFK,Nes,h,Kir}, and Thomas \cite{Th} -- for projective and affine GIT-quotients and moduli of vector bundles over Riemann surfaces see Newstead's book and article \cite{new, new1}.

Let the reductive Lie group $G$ act linearly on a finite dimensional complex vector space $V$. A point $ v\in V \backslash\{0\}$ is called {\em unstable} for the group action if $0\in V$ is contained in the closure of the orbit of $v$, otherwise it is called {\em semistable}. A semistable point is called {\em polystable} if its $G$-orbit is closed, and {\em stable} if it is polystable with a finite isotropy group. (The point $0\in V$ is not part of this classification -- in the present situation it is meaningful to define it as polystable).

The existing GIT-quotient $V\s2 G$ of the affine space $V$ is an affine space equipped with the algebra $\C[V]^G$ of $G$-invariant regular functions. Furthermore there exists a morphism $p:V \to V\s2 G$.

\subsection{Restricted actions of reductive groups on complex spaces}\label{se:restrac}
Actions of reductive groups on analytic space germs have been investigated (cf.\ G.~Müller's article \cite{mue}), however, a global approach is necessary. Applications of the Kempf-Ness Theorem \cite{kn1} will be needed first:

Assume that the group $G$ is the complexification of a maximal compact subgroup $K$, and let  $\rho: G\to GL(V)$ be a representation. By assumption the vector space $V$ carries a hermitian inner product with norm $m$ and corresponding moment map $\mu$ for the action of $K$ on $V$. In this context the point $0\in V$ is considered to be in $V^{ps}$ (with $\mu(0)=0$). 
\begin{theorem*}[{Kempf-Ness \cite{kn1}}]\strut
\begin{itemize}
	\item [(i)] The polystable locus $V^{ps}$ is equal to  $G\cdot\mu^{-1}(0)$.
	\item[(ii)] Every $G$-orbit  in $V^{ps}$ contains only one $K$-orbit of $\mu^{-1}(0)$.
	\item[(iii)] In particular the GIT-quotient $V\s2 G$ can be identified with the set theoretic quotient $V^{ps}/G$,
	\item[(iv)] and the space $\mu^{-1}(0)/K$ equipped with the quotient topology is homeomorphic to the GIT-quotient $V\s2 G$.
\end{itemize}
\end{theorem*}
We note the following consequence:

\begin{remark}\label{re:actU}
Let $U\subset V$ be an open ball around $0\in V$ with respect to the norm $m$.
Assume that $K$ is the group of unitary automorphisms of $V$. Then the image of $U$ under $p:V\to V\s2 G$ is an open neighborhood of $p(0)\in V\s2 G$.
\end{remark}

\begin{proof}
	The set $U$ is $K$-invariant and $\mu^{-1}(0)\cap U$ is open in $\mu^{-1}(0)$. 	
\end{proof}

For moduli theoretic applications the following situation is of interest. The existence of a complex structure on the quotient requires the existence of an action of the complex Lie group $G$ on a certain space, where initially only an action of the group germ of $G$ on the germ of $S$ at $0$ exists.

The situation resembles the action of $\C^*$ on $\C$ -- the restriction to the unit disk $\Delta\subset \C$ does not yield an action. Lacking an established notion we will speak of a ``restricted action'' -- also, if an action of $G$ on some total space still has to be constructed. For the intended application, a linear action of $G$ on the vector space $V$ is already given -- the restriction to an open neighborhood $U$ of the origin in $V$ is a restricted action, and finally there is a restricted action to a closed analytic subset $0\in S\subset U$ to be defined.
\begin{definition}[Restricted group action]\label{de:rgp}
Let a reductive group $G$ act linearly on a complex vector space $V$. Let $0\in S\subset U$ be a closed analytic subset of an open subset $U\subset V$.

A {\em restricted action} of $G$ on the pair $(S,U)$ is defined by the property that for all $s \in S$ and $g\in G$ with $g\cdot s\in U$ the point $g\cdot s$ is contained in $S$.

The induced equivalence relation on $S$ is given by
$$
s'\sim s'' \text{ if and only if there exists }g\in G \text{ such that }s''= g\cdot s'.
$$
\end{definition}

A restricted group action on a pair $(S,U)$ can be globalized in the following sense. The globalization is related to the complexification in the sense of Heinzner \cite{he}. We will give the short proof.

\begin{theorem}\label{th:glob}
	There exists a closed $G$-invariant analytic subset $\wt S:= \bigcup_{g\in G}g\cdot S$ of the $G$-invariant open set  $\wt U:= \bigcup_{g\in G}g\cdot U \subset V$ such that $\wt S\cap U=S$.
\end{theorem}
In general $\wt U \subsetneq V$, and $\wt S$ need not be affine. In fact, by Remark~\ref{re:actU} the set $W=p(U)$ is an open neighborhood of $p(0)$.
\begin{notation*}
	Given a restricted group action on a pair $(S,U)$ the tilde notation such as $\wt S\subset \wt U$ will be consistently used for the globalized spaces.
\end{notation*}
\begin{proof}[Proof of Theorem~\ref{th:glob}]
	Let $g\in G$ be given so that $U\cap g\cdot U\neq\emptyset$. The set $S\cup g\cdot S$ is locally analytic in $U\cup g\cdot U$. It is an analytic subset of $U\cup g\cdot U$ if it is closed. Let $\ol S$ be the closure in $U \cup g\cdot U$. Take $x\in \ol S\backslash S$. Then $x\in g\cdot S$. Near this point, on a certain neighborhood $U(x)$ of $x$ the set $g\cdot S$ is given by certain holomorphic equations $f_1=\ldots=f_m=0$. Consider $(g^{-1}\cdot U(x))\cap U$ and the  pulled back equations. Since $S$ is invariant in the restricted sense, these equations must vanish at $x$ so that $x\in g\cdot S$. The argument is symmetric in $U$ and $g\cdot U$, and the same argument holds for sets $g_1\cdot U$ and $g_2\cdot U$.
	
	Finally, the analyticity of $\cup_{g}g\cdot S\subset \wt U$ follows, if this set is closed in $\wt U$. Let $s\in \ol{\cup_{g}g\cdot S}\subset \wt U$, say $s=\lim g_{j}\cdot s_j\in g_{j_0}\cdot U$. Then $g^{-1}_{j_0}\cdot s =\lim_j g^{-1}_{j_0}g_{j}\cdot s_j \in U\cap S$ so that $s\in g_{j_0}S \subset \wt S$.
\end{proof}

In the algebraic setting an immediate argument shows that the ideal sheaf of functions that vanish on the image $p(\wt S)$ in the open subset $W=p(\wt U)\subset  V\s2 G$ is coherent as a certain ideal of $G$-invariant functions on preimages of open subsets. In the analytic setting a different approach is needed, namely first prove the analyticity of $p(\wt S)$ by Kuhlmann's generalization of Remmert's proper mapping theorem (cf.\ also the book by Whitney) \cite{ku,w,re}.
\begin{definition}
	Let $f:X \to Y$ be a continuous mapping of locally compact, Hausdorff topological spaces with countable topology. The map $f$ is called {\em semi-proper}, if for any point of $Y$ there exists an open neighborhood $W$ and a compact set $L\subset X$ such that
	$$
	f(L)\cap W = f(X)\cap W.
	$$
\end{definition}
Obviously restrictions of semi-proper maps to closed subsets need not be semi-proper in general.
\begin{theorem*}[{Semi-proper mapping theorem, \cite{ku,w,re}}]
Let $f:X\to Y$ be a semi-proper holomorphic map of reduced complex spaces. Then $f(X)\subset Y$ is an analytic subset.	
\end{theorem*}

The following situation is given,
$$
\xymatrix{
	\wt S  \ar[rd]  \ar@{^{(}->} [r]& \wt U \ar[d]^{p|\wt U} \ar@{^{(}->} [r] & V \ar[d]^p\\
	& W \ar@{^{(}->} [r]& V\s2 G
}
$$
where $p$ is rational (with restriction to a classically open subset $\wt U$).
As in Theorem~\ref{th:glob} the set $U$ is a convex neighborhood of $0$ of the form
$$
U=U_c=\{ x; m(x)< c  \}
$$
for some $c>0$, and $\wt U$, $W=p(U)=p(\wt U)$ in the above sense.

\begin{theorem}\label{th:spr}
	The canonical holomorphic maps $p|\wt U : \wt U \to W$ and the restriction $p|\wt S : \wt S \to W$ are semi-proper. In particular $p(\wt S)\subset W\subset V\s2 G$ is a closed analytic subset of $W$. 	
\end{theorem}

\begin{proof}
	The Kempf-Ness Theorem will be applied. Let $w\in W$. Let $V^{ps}$ denote the polystable locus as above. There exists $u\in V^{ps}$ with minimal norm within its orbit such that $p(u)=w$. It follows from the construction of $\wt U$ and the convexity of $U$ that $u\in U$.  Choose $0<a<b<c$ such that  $u\in U_a$, and accordingly $w$ is contained in the open neighborhood $W_a=p(U_a)$. Pick the compact set $L= \ol{U_b}$. Then $p(L)\cap W_a = p(\wt U)\cap W_a$.
	
	Now the semi-properness of $p|\wt S:\wt S \to W$ is shown. Let $w\in W$. Let $0<a<b<c$ such that $ w \in W_a=p(U_a) \subset W = p(U)$.  We claim that
	$$
	p(\wt S)\cap W_a= p( \ol{U_b}\cap \wt S)\cap W_a.
	$$
	Namely, let $p(\wt s) \in W_a$ for some $\wt s\in \wt S$. Then there exists some $\wh s\in V^{ps}$ with $\ol{G\cdot \wt s} = G \cdot \wh s$ and $\wh s$ of minimal norm within its $G$-orbit. Hence $\wh s \in U_a\subset \ol{U_b}$. Furthermore $\wh s \in \ol{G\cdot \wt s}\subset \ol{\wt S} = \wt S$ by Theorem~\ref{th:glob}, Altogether $p(\wt s)\in p(\ol{U_b}\cap\wt S)$, where $\ol{U_b}\cap\wt S= \ol{U_b}\cap S$ is compact.
\end{proof}
Let $p:V\to \VsG$ be the quotient of a smooth (or maximal) affine space $V$ by a reductive group, $W\subset \VsG$ be an open Stein subset, and $\wt U =p^{-1}(W)$. Let $\wt S \subset \wt U$ be a $G$-invariant closed subset whose image $p(\wt S)$ is a closed analytic subset of $W$. Then
$$
\wt S \hookrightarrow \wt U \hookrightarrow V
$$
is $G$-equivariant. (This is the motivation for the notation $(\wt S \hookrightarrow \wt U \hookrightarrow V)$ as opposed to $(\wt S \subset \wt U \subset V)$.)

\begin{definition}\label{de:anGIT}
The triple $(p(\wt S) \hookrightarrow W \hookrightarrow V\s2 G)$ is called analytic GIT-quotient of $\wt S$ by $G$, and denoted by
$$
(\wt S\san2 G\hookrightarrow \wt U\san2 G\hookrightarrow \VsG).
$$
\end{definition}
Note that $p|\wt S:\wt S \to \wt S\san2 G$ identifies all points in a $G$-orbit  (and its closure).
\begin{remark}
By Theorem~\ref{th:glob} and Theorem~\ref{th:spr} any restricted group action in the sense of Definition~\ref{de:rgp} gives rise to an analytic GIT-quotient.
\end{remark}

\subsection{Complex structure for analytic GIT-quotients}\label{se:anGIT}
An analytic GIT-quotient $\wt S\san2 G$ according to Theorem~\ref{th:spr} carries the natural complex structure of an analytic subset of the open subset $\wt U\san2 G\subset V\s2 G$.

Let $V$ be a smooth affine space together with the action of a reductive group $G$ such that all points of $V$ are semistable. Denote by $p: V \to V\s2 G$ the GIT-quotient. It is known (cf.\ articles by Neeman \cite{nee,nee1}, Roberts \cite{r}, Snow \cite{sn}, and Thomas \cite{Th}) that due to the normality of $V\s2 G$ the (analytic) structure sheaf satisfies
\begin{equation}\label{eq:Ginv}
\cO_\VsG(W')= \cO^G_V(p^{-1}(W'))
\end{equation}
for any (Stein) open subset $W'\subset \VsG$, where a holomorphic function $\varphi$ on a subset $W'$ of the GIT-quotient corresponds to the $G$-invariant function $\varphi\circ p$ on $p^{-1}(W')\subset \wt S$.

The aim is to show that holomorphic functions on open subsets of the quotient $\wt S\san2 G$ are exactly the $G$-invariant holomorphic functions on the preimages under $p$.

Let $W\subset \VsG$ be a Stein open subset, and $\wt U=p^{-1}(W)\subset V$. Let $\wt S\subset \wt U$ be a closed, $G$-invariant, analytic subset. The coherence of the ideal sheaf of $G$-invariant functions vanishing on $\wt S$ can be shown directly.

\begin{remark}\label{pr:idS}
Let $p(\wt S)=\wt S\san2 G \subset W= \wt U\san2 G \subset \VsG$ be an analytic GIT-quotient. Then for the (coherent) vanishing ideals $\cI_{p(\wt S)}\subset \cO_\VsG|W$ and $\cI_{\wt S}\subset \cO_V|\wt U$ the following holds:
$$
\cI_{p(\wt S)}(W')= \cI_{\wt S}(p^{-1}(W'))^G,
$$
where $W'\subset W$ is an open Stein set, and where $\cI_{\wt S}(p^{-1}(W'))^G$ denotes $G$-invariant holomorphic functions that vanish on $\wt S$.
\end{remark}
\begin{proof}
Let $\varphi \in \cI_{p(\wt S)}(W')$. By \eqref{eq:Ginv} the function $\varphi\circ p$ is $G$-invariant and vanishes on $p^{-1}(W'\cap p(\wt S))$.

Conversely, let a $G$-invariant holomorphic function on $p^{-1}(W')$ be given that vanishes on $p^{-1}(p(\wt S)\cap W')$. 	By \ref{eq:Ginv} it is of the form $\varphi \circ p$, hence it descends to $\varphi$ with $\varphi|W'\cap p(\wt S)=0$.
\end{proof}

In this sense by Proposition~\ref{pr:idS}, given a $G$-invariant analytic set $\wt S \subset \wt U$ with analytic image $p(\wt S)\subset W$ the short exact sequence
$$
0 \to \cI_{p(\wt S)} \to \cO_\VsG|W \to \cO_{\wt S\san2 G} \to 0
$$
implies the exactness of
\begin{equation}\label{eq:str}
0 \to  \cI_{\wt S}(p^{-1}(W'))^G \to \cO_V(p^{-1}(W'))^G \to \cO_{\wt S\san2 G}(W') \to 0.
\end{equation}
Since $W'\subset V\san2 G$ is an open Stein subset, \eqref{eq:str} implies the existence of a natural injective map
\begin{equation}\label{eq:oSG}
 \iota:\cO_{\wt S\san2 G}(W') \hookrightarrow \cO_{\wt S}(p^{-1}(W'))^G.
\end{equation}

\begin{theorem}\label{th:invfct}
  The embedding \eqref{eq:oSG} is an isomorphism.
\end{theorem}

If we restrict the claim to the ``normal category'' i.e.\ to weakly normal spaces (or to the weak normalization of a space $\wt S$), then the claim of the theorem follows in an elementary way:  A reduced complex space is called weakly normal if continuous holomorphic functions on the regular locus of an open subset extend holomorphically to the given open subset. The condition is equivalent to saying that any continuous, weakly holomorphic function (i.e.\ holomorphic outside a nowhere dense analytic subset) is holomorphic. We use the notation $\wh S \to S$ for the weak normalization (cf.\ the book by L.~Kaup and B.~Kaup \cite[\S 72]{K-K}). 	If $S$ is weakly normal, then so is $\wt S$. The analytic GIT-quotient $\wt S\san2 G$ exists and by \cite[\S 72.4]{K-K} again is weakly normal. Since $G$-invariant holomorphic functions on $\wt S$ descend to continuous functions that are holomorphic on the locus where $p$ is of maximal (differential rank), the structure sheaf of the weak normalization of $\wt S\san2 G$ satisfies
$$
\cO_{\wh{\wt S \san2 G}} (W') = \cO_{\wh{\wt S}}(p^{-1}(W'))^G
$$
for Stein open subspaces $W'\subset W$.

Now the general case of Theorem~\ref{th:invfct} will be proven.
\begin{proposition}
  Let $W'\subset V\s2 G$ be a $K$-invariant Stein open set, e.g.\ let $W'=p(U)$, where $U\subset V$ is an open ball around $0\in V$. Then $p^{-1}(W')\subset V$ is a Stein open subset.
\end{proposition}

\noindent
  The {\em proof} is an immediate consequence of Heinzner's complexification  theorem -- it follows that the set $p^{-1}(W')$ is equal to the complexification $W'^{\C}$ in the sense of \cite{he}. \qed
\medskip

In \cite{r} Roberts studied coherent sheaves $\mathcal F$ on a space $V$ with $G$-actions. These give rise to sheaves $p^G_*\mathcal F$ on $\VsG$ defined by
$$
(p^G_*\mathcal F)(W'):= \cF(p^{-1}(W')^G\text{ for Stein open subsets }W'
$$
(cf.\ also the article by Heinzner and Loose \cite{h-l}).

Roberts constructs an averaging operator for coherent $\cO_{\wt U}$-modules $\cF$.
$$
L: \cF_{\wt U}(p^{-1}(W'))  \longrightarrow \cF_{\wt S\san2 G}(W')^G
$$
by means of integration over the compact group $K$ that is a projection onto $G$-invariant sections.  By \cite[Theorem~3.1]{r} the coherence of $p_*\cO_{\wt S}^G$ follows, and \eqref{eq:Ginv},\eqref{eq:str} imply the proof of Theorem~\ref{th:invfct}. \qed

\section{Slice theorem for analytic GIT-quotients}
Snow showed an analytic slice theorem for the action of a reductive group on a Stein space together with the existence of a categorical quotient in \cite{sn}. Here the situation is simpler than in the general case.

The action of a reductive group $G$ on an affine space $V$ is already given, and restricted to a closed subset of an open (Stein) subspace. We need the direct application to spaces of the form $\wt S\subset \wt U\subset V$.

A slice for the action on an affine space is given by Luna's theorem \cite{Lu} (see also the exposition by Drézet \cite{d}).

\subsection{Algebraic case}
Let the reductive group $G$ act on an affine space (or vector space) $V$. If the $G$-orbit of a point $s\in V$ is closed, then by Matsushima's theorem the stabilizer subgroup $G_s$ is reductive again.

Let $Y$ be a $G_s$-invariant affine subspace containing $s$, and $G\times Y \to V$ be the canonical $G$-equivariant map, where the left action of $G$ on itself defines an action on $G\times Y$. The left action of $G_s$ on $G$ together with the action of $G_s$ on $Y$ is given by $\gamma\cdot(g,w)=(g\gamma^{-1}, \gamma w)$ for $\gamma \in G_s$, $g\in G$, and $w\in Y$. It defines a GIT-quotient  $(G\times Y)\s2 G_s $ for which the notation $G\times_{G_s}Y$ is common usage. Now $G\times_{G_s}Y\s2 G \simeq Y\s2 G_s$. Altogether there is a commutative diagram
$$
\xymatrix{
	G\times Y \s2  G_s =   G\times_{G_s} Y   \ar[d]\ar[r]^{\hspace{12mm}\psi}  &  V_0 \subset V \ar[d]\\
	Y\s2  G_s \simeq ( G \times_{G_s}Y)\s2 G \ar[r]^{\hspace{10mm}\ol\psi} & V_0\s2 G \subset V\s2 G
}
$$
where $\psi$ and $\ol\psi$ are induced by the action of $G$. The spaces $V_0$ and $V_0\s2 G$ are the images of $\psi$ and $\overline\psi$ respectively in the sense below.

\begin{theorem*}[{Luna, \cite{Lu}\cite[Thm.\ 5.3, 5.4]{d}}] Let $G$ be a reductive group acting on an affine variety $V$ with closed orbit $G\cdot s$. Then there exists a slice $Y$ through $s$, i.e. a $G_s$-invariant affine subvariety such that the map $\psi$ is \'etale onto an open, $G$-invariant subvariety $V_0\subset V$, and such that $\overline\psi$ is \'etale onto $V_0\s2 G$. If $V$ is smooth at $s$ so is $Y$.
	\vskip0mm
	Let $V'\subset G\times_{G_s}Y$ be a closed $G$-invariant subvariety. Then there exists a closed $G_s$-invariant subvariety $Y'$ of $Y$ such that $V'= G\times_{G_s} Y'$.
\end{theorem*}

Note that the dimension of the isotropy subgroups of $G$ is not assumed to be constant.

\subsection{Application to analytic GIT-quotients}
For any analytic application the property {\em étale} will be replaced by {\em locally biholomorphic}, and the (analytic) structure sheaves of analytic GIT-quotients such as $p: \wt S \to \wt S\san2 G$ of reduced complex spaces consist of holomorphic $G$-invariant functions on pull-backs under the projection $p$.

Let $\wt S \hookrightarrow \wt U \hookrightarrow V$ be given as in Section~\ref{se:anGIT} where $V$ is a smooth affine space. In particular, referring to Definition~\ref{de:anGIT}, let $s\in \wt S$ be a point with a closed $G$-orbit, and let $Y$ be a slice through $s$ for the action of $G$ on $V$ according to the slice theorem. Without loss of generality let $\wt S\subset V_0$.

In the process the space $\wt U$ is replaced by a smaller saturated open subspace that contains $s$ and that is of the form $p^{-1}(W)$. Accordingly the algebraic slice $Y$ will be replaced by $\wt Y = \wt U\cap Y$. The analytic quotients $(G\times \wt Y))\san2 G_s$, $\wt Y\san2 G_s$, and $\wt Y\san2 G$ exist, and the structure sheaves consist of $G_s$- and $G$-invariant functions according to Section~\ref{se:anGIT}. Altogether there exist analytic quotients and locally biholomorphic, surjective maps $\psi$ and $\ol \psi$
\begin{equation}\label{eq:arr}
\begin{array}{ccc}
(G\times Y)\s2 G_s &\stackrel{\psi}{\longrightarrow} & V_0\\
\cup && \cup\\
(G\times \wt Y))\san2 G_s & \longrightarrow &\wt U \\
\downarrow p && \downarrow p \\
\wt Y\san2 G_s & \to & \wt U \san2 G\\
\cap&&\cap\\
Y\s2 G_s & \stackrel{\ol \psi}{\longrightarrow}&V_0 \s2 G.
\end{array}
\end{equation}
With the above notation the following theorem holds.
\begin{theorem}[Application to analytic GIT-quotients]\label{th:appLuna}
Let the reductive group $G$ act on a smooth (or weakly normal) affine space $V$, and as above let $S\subset U\subset V$ be a $G$-invariant analytic subset with $0\in S$. Then for points $0\neq s\in \wt S$ the following holds: Let the orbit $G\cdot s \subset V$ be closed. Then
\begin{itemize}
  \item[(i)] There exists the slice $\wt S\cap \wt Y$ for the action of $G_s$ on $\wt S$, namely a locally biholomorphic map from $(G\times (\wt S \cap \wt Y))\san2 G_s$ to $\wt S$.
  \item[(ii)] The action of $G$  defines an analytic GIT-quotient
  $$
  (G\times(\wt S \cap \wt Y) \san2 G_s)\san2 G  \simeq (\wt S \cap \wt Y)\san2 G_s,
  $$
  and a locally biholomorphic map to $\wt S\san2 G$
\end{itemize}
$$
\begin{array}{ccc}
G\times(\wt S\cap \wt Y)\san2 G_s & \stackrel{\psi}{\longrightarrow} & \wt S\\
\downarrow p && \downarrow p \\
(\wt S \cap \wt Y) \san2 G_s & \stackrel{\ol\psi}{\longrightarrow} & \wt S\san2 G .
\end{array}
$$
\end{theorem}
\begin{proof}
	Consider the multiplication map $G\times \wt Y\to \wt Y$. If $g\cdot y\in \wt S$, then obviously $y\in \wt S$ so that $y\in \wt S \cap \wt Y$, hence $\psi^{-1}(\wt S)= G\times (\wt Y\cap \wt S)\san2 G_s$. This proves (i). The same argument and \eqref{eq:arr} yield (ii). Observe that the analytic GIT-quotient are images under a projection map equipped with the sheaves of $G_s$-invariant functions.
\end{proof}

\section{Deformations of holomorphic vector bundles}
We will fix some notation now. Let  $X$ be a compact complex manifold, which will be equipped later with a \ka\ form $\omega_X$. A holomorphic family $(\cE_s)_{s\in S}$ of holomorphic vector bundles on $X$ parameterized by a (reduced) complex space $S$ is given by a holomorphic vector bundle $\cE$ over $X\times S$ such that $\cE_s=\cE|X\times\{s\}$. Deformations $\xi$ of a fixed holomorphic vector bundle $E$ are defined over complex spaces $(S,s_0)$ with a distinguished point $s_0\in S$ (or a ``pointed'' complex space). Such an object consists of a holomorphic family $\cE$ over $S$ together with an isomorphism $\chi: E \stackrel{\sim}{\longrightarrow}\cE_{s_0}$. Accordingly an isomorphism of two deformations of a bundle $E$ over the same space is an isomorphism of holomorphic families that is compatible with the given identifications of central fibers with $E$. Consequently, a change of $\chi$ usually generates a non-isomorphic deformation (unless $\chi$ can be extended to the whole family).

During most arguments that involve deformations a parameter space has to be replaced by a neighborhood of the distinguished point. In this sense space germs are being used as parameter spaces for deformations.

Let $\beta:(R,r_0)\to (S,s_0)$ be a holomorphic map of spaces with distinguished points, and $\xi$ a deformation of $E$ over $(S,s_0)$. The base change map assigns to $\xi$ a deformation $\beta^*\xi$ over $(R,r_0)$, which is given by the pull-back $\beta^*\cE$ of the corresponding family together with the induced isomorphism of $E$ and the fiber $\beta^*\cE|X\times \{r_0\}=\cE_{s_0}$.

{\em Infinitesimal deformations} of a bundle $E$ are by definition deformations over the ``double point'' $D=(0,\C[t]/(t^2))\subset (\C,0)$, and tangent vectors $v\in T_{s_0}S$ correspond exactly to holomorphic maps $\underline{v}:D \to S$ sending the underlying point to $s_0\in S$. The space of isomorphism classes of infinitesimal deformations of $E$ can be identified with the complex vector space $H^1(X, End(E))$. In this set-up, given a deformation $\xi$,  the \ks\ map
$$
\rho:T_{s_0}S\to H^1(X, End(E))
$$
sends a tangent vector $v$ to the pull-back $\underline{v}^*\xi$, which is the ``restriction of the given deformation to the direction given by the tangent vector $v$''.

A deformation $\xi$ is called {\em complete}, if any other deformation is of the form $\beta^*\xi$, where $\beta$ is a base change map -- it is called {\em semi-universal}, if in addition the derivative $T_{r_0}\beta$ at the distinguished point is uniquely determined.

More precisely, let $E$ be a polystable holomorphic vector bundle on \break $(X,\omega_X)$, and $\xi$ a semi-universal deformation of $E$ over a complex analytic space $(S,s_0)$. In order to have a halfway simple notion deformations are restricted to the category of reduced spaces first, and $(S,s_0)$ will always denote a {\em reduced} complex space. Concerning the \ks\ map the restriction to reduced complex analytic spaces yields values in a linear subspace of $V=H^1(X,End(E))$.

\section{Application to deformations of polystable vector bundles}\label{se:appdefvb}
Let again $(X,\omega_X)$ be a compact \ka\ manifold. For a polystable vector bundle the base of a semi-universal deformation $\xi$ in general contains positive dimensional local analytic subsets, namely parts of $G$-orbits of semistable (but not polystable points). Also, in general, stable points give rise to local analytic sets, where the fibers are isomorphic. This fact can be illustrated as follows. If the tangent cohomology $H^q(X,End(E))$ vanishes for $q\geq 2$ then for $E=\cE_0$ polystable and $\cE_s$ stable for some $s$ near $s_0$ we have $0< h^0(X,\cE_{s_0})-h^0(X,\cE_s) =h^1(X,\cE_{s_0})-h^1(X,\cE_s)$. Now $\xi$ induces a complete deformation of $\cE_s$ (``openness of versality''), but the dimension of the smooth space $S$ is too large for universality at $s\neq s_0$ so that the deformation $\xi$ restricted to a positive dimensional local analytic set has to be trivial. This situation needs to be handled by an (analytic) GIT approach.

\begin{theorem}[{\cite{bu-sch}[Theorem~1]}]\label{th:Psi}
	For any polystable vector bundle on a compact \ka\ manifold there exists a semi-universal deformation with a parameter space $(S,s_0)$ such that the action of the group of holomorphic automorphisms $Aut(E)$ on the tangent space of $S$ at $s_0$ extends to an action on the space germ of $S$ at $s_0$ that is compatible with the holomorphic family of vector bundles.
\end{theorem}

The space $H^2(X,End(E))$ which contains the obstructions to extending the bundle $E$ to a family over a whole neighborhood of $0$ in  \break $H^1(X,End(E))$ is equipped with a compatible action of $G=Aut(E)$.

In \cite{bu-sch} the space $S$ was realized as an analytic subspace of an open neighborhood $U$ of $s_0=0$ in the space of infinitesimal deformations \break $H^1(X, End(E))$ of the form
\begin{equation}\label{eq:Psi}
S=\Psi^{-1}(0)\subset U \text{ where }\Psi: U \to H^2(X, End(E))
\end{equation}
was constructed in \cite{bu-sch} as an $Aut(E)$-equivariant map in the restricted sense of Definition~\ref{de:rgp}. For a suitable chosen open set $U$ all fibers of points from $S$ correspond to semistable vector bundles.

\textit{
	In other words, the action of $G=Aut(E)$ on $V$ gives rise to a restricted action on $(S,U)$ in the sense of Section~\ref{se:anGIT0}.
}
\begin{theorem}\label{th:localmain}
  Let a polystable vector bundle $E$ on a compact \ka\ manifold $(X,\omega_X)$ be given. Then a semi-universal deformation of $E$ over a space $0\in S\subset U\subset V$  together with the reduced action of $G=Aut(E)$ on the space $V$ of infinitesimal deformations gives rise to an analytic GIT-quotient $\wt S\san2 G$ which is a closed reduced subspace of an open subset  $\wt U\san2 G$ of the GIT-quotient $V\s2 G$.
\end{theorem}

The following correspondence will also be crucial:

\begin{theorem}[{\cite[Thm.~3]{bu-sch}}]\label{th:stab}
A class $s\in S=\Psi^{-1}(0)\subset H^1(X,End(E))$
is (poly)stable with respect to the action of $Aut(E)$
if and only if the corresponding bundle $\cE_s$
is (poly)stable with respect to the given \ka\ metric.
\end{theorem}

For the property of semi-stability such an analogy does not exist: Let $V^{ss}$ denote the semi-stable locus. In general there may be points $s\in S\backslash V^{ss}$ that correspond to semi-stable bundles (cf.\ \cite[\S 9]{bu-sch}). A change of the affine space can turn non-semistable points into semistable ones.

In Section~\ref{se:WP} we will need the definition of $\Psi$. We will give a few details from the construction in \cite{bu-sch}. Let a polystable bundle $(E,\ol\pt_0)$ be given, equipped with a \he\ metric. According to \cite[Prop.~2.6]{bu-sch} any $End(E)$-valued, harmonic $(0,1)$-form $\alpha$ from $U$ is assigned to in a unique way a form
$$
a''= \alpha + \ol\pt^*_0 \beta
$$
with estimates, where $\beta$ is a certain $End(E)$-valued (0,2)-form such that
$$
\ol\pt_0 a'' + a''\we a''
$$
is $\ol\pt^*_0$-closed. Let $\Pi$ be the harmonic projection on the space of (0,2)-forms. Then
$$
\Psi(\alpha)= \Pi(a''\we a'').
$$
The map $\Psi$ is holomorphic, and the integrability condition for $a''$ is satisfied, if and only if $\Psi(\alpha)=0$, and the set $S=\Psi^{-1}(0)\subset U\subset  H^1(X,End(E))$ is the parameter space of a semi-universal deformation.

\begin{remark}\label{re:semi-c}
The family of semi-connections $a''$ over $U$ together with the flat semi-connection in horizontal direction is a semi-connection on $E\times U$ over $X\times U$. Over $X\times S$ this semi-connection is integrable.
\end{remark}

\subsection{Change of affine space}
In principle the special role of the origin in $V$ with respect to the group action is not intended. In this affine setting points that are unstable have the same status as points that are semistable but not polystable. One can use the following auxiliary affine space.

Set $V_0=\C\times V$.
\begin{lemma}
A given action of $G$ on $V$ is extended to $V_0$ by the identity on $\C$, i.e.\ $g\cdot (\zeta,v):= (\zeta,g\cdot v)$. Then the following holds (again after dividing by the ineffectivity kernel). Let $(\zeta,v)\in V_0$.
  \begin{itemize}
    \item[(i)] $G_{(\zeta,v)}=G_v$
    \item[(ii)] $G\cdot (\zeta,v) = \{\zeta\}\times G\cdot v$
    \item[(iii)] Let $\zeta\neq 0$, $v\neq 0$, then $(\zeta,v)$ is $G$-(poly)stable, if and only $v$ has this property. Let $\zeta\neq 0$, then $(\zeta,0)$ is $G$-polystable, and all $(\zeta,v)$ are $G$-semistable points.
  \end{itemize}
\end{lemma}
Observe that $\C[V_0] = \C[\zeta]\otimes \C[V]$, and accordingly $\C[V_0]^G= \C[\zeta]\otimes \C[V]^G$ so that $V_0\s2 G= \C\times (V\s2 G)$.

Nevertheless in the sequel the original spaces will be used.

\subsection{Analytic GIT-quotient for parameter spaces of holomorphic vector bundles}
Given a semi-universal deformation of a polystable vector bundle, the parameter space $S$ can be chosen so that the triple
$$
(S\hookrightarrow U \hookrightarrow V)
$$
satisfies the assumptions of Theorem~\ref{th:localmain}.

Eventually all points of $\wt S\san2 G$ correspond to holomorphic vector bundles or equivalence classes of holomorphic vector bundles as in the case of algebraic GIT-spaces.

\begin{proposition}\label{pr:trfree}
	Let $\cE \to  X\times Q $ be a holomorphic family of vector bundles over a compact \ka\ manifold $(X,\omega_X)$ parameterized by an irreducible, reduced space  $Q$.
	Then there exists an  analytic subset $A\subset Q$ such that all bundles $\cE_s$ for $s\in Q\backslash A$ possess a universal deformation.
\end{proposition}
\begin{proof}
Let $\cF$ be the sheaf of trace-free holomorphic endomorphisms of $\cE$. Then the above locus $A$ is contained in the support of $pr_*\cF$, where $pr:X\times S\to S$ denotes the canonical projection.
\end{proof}

\begin{remark}\label{re:pslocus}
The set of polystable points that are not stable in the analytic GIT-quotient $\wt S\san2 G$ is a closed analytic subspace.
\end{remark}
\begin{proof}
  Given a deformation of a polystable bundle over a space $(S,0)$ the set $A\subset S $ from Proposition~\ref{pr:trfree} is $G$-invariant in the restricted sense and gives rise to a GIT-subspace $\wt A\san2 G\subset \wt S\san2 G$, whose points represent polystable, non-stable vector bundles.
\end{proof}

The following fact will be needed. Let $\Psi$ and $S$ be given as in \eqref{eq:Psi}.
\begin{proposition}[{\cite{bu-sch}[Corollary 4.6]}]\label{pr:Cor46}
	Let $p>2\,{\mathrm dim}(X) $. Then the (pointwise) quotient %
    $S/G$ can be identified with $\mathcal F/\mathcal G$, where $\mathcal{F}$ denotes the space of $L^p_1$ integrable connections near the given  \he\ connection $d_0$ on $E$, and $\mathcal G$ denotes the corresponding complexified gauge group: parameters with isomorphic fibers are actually contained in the same $G$-orbit.
\end{proposition}

Let $S=\Psi^{-1}(0) \subset U$ be the base of a semi-universal deformation of a polystable bundle. Two points $s', s''$ of $S$ will be identified, $s'\sim s''$, if the fibers $\cE_{s'}$ and  $\cE_{s''}$ are isomorphic. By Proposition~\ref{pr:Cor46} this is exactly the case if $s'=g\cdot s''$ for some $g\in G$. Conversely, if $s', s'' \in \wt S$ then there exist $g',g''\in G$ such that $g'\cdot s', g''\cdot s''\in S$. If these points of $S$ are equivalent, $s'= g \cdot  s''$ holds for some $g\in G$
so that the topological spaces $\wt S/G$ and $S/\!\sim$ can be identified.

After applying the methods from Section~\ref{se:anGIT0} to the reduced action of $G$ to this standard situation the following holds.

\begin{lemma}\label{le:gitiso1}
	Let $\xi$ be a semi-universal deformation of a polystable bundle over $(S,0)$, and $\eta:(S,0)\to (S,0)$ a holomorphic map (of germs) such that $\eta^*\xi \simeq \xi$. Then $\eta$ is an isomorphism, whose derivative at the origin is the identity. The map $\eta$ defines a map  $S_1 \to S_2$, where $S_1, S_2\subset S$ are neighborhoods of the origin  that extends to a map $\wt S_1 \to \wt S_2$ and descends to an isomorphism of analytic GIT-quotients
	$$
	{\wt S_1  \san2 G} \stackrel{\sim}{\longrightarrow} {\wt S_2  \san2 G}.
	$$
\end{lemma}
\begin{proof}

The property of $\xi$ being a semi-universal deformation implies that $\eta$ is an isomorphism with $T_0\eta = id_{T_0S}$. Represent $\eta$ by an isomorphism denoted by the same letter $\eta:S_1\to S_2$ of suitable neighborhoods of the origin. Note that the construction of the analytic GIT-quotient contains the fact that the maps $\wt S_j \to \wt S_j\san2 G$ etc.\ are still surjective when restricted to $S_j \subset \wt S_j$. Due to the construction of $\eta$ and Proposition~\ref{pr:Cor46} the map $\eta$ is constant on closures of $G$-orbits, in particular it is $G$-invariant. The claim now follows by Theorem~\ref{th:invfct}.
\end{proof}

Let again a semi-universal deformation $\xi$ of a polystable bundle over $(S,0)$ be given, and the induced action of $G$ on $(S\hookrightarrow U \hookrightarrow V)$. Let $s_0\in S$ correspond to a polystable bundle $\cE_{s_0}$.

Let $R\subset U_R\subset V_R$ denote the base of a semi-universal deformation $\zeta$ of ${\cE}_{s_0}$ with restricted group action of $H=Aut({\cE}_{s_0})$ on $(R\hookrightarrow U_R \hookrightarrow V_R)$ for $V_R = H^1(X,End({\cE}_{s_0}))$.

Now there exist base change maps of space germs  $\alpha: (S,s_0) \to (R,0)$ and $\beta : (R,0) \to (S,s_0)$ such that $\alpha^* \zeta = \xi_{s_0}$ and $\beta^*\xi_{s_0} = \zeta$, because of completeness and semi-universality of $\xi_{s_0}$ and $\zeta$ respectively. By Lemma~\ref{le:gitiso1} the map $\alpha\circ \beta$ induces an isomorphism of neighborhoods of the distinguished point which descends to an isomorphism (near the distinguished point) of  $\wt R\san2 H$.

\begin{lemma}\label{le:betadesc}
	The map $\beta$ descends to a holomorphic map
	$$
	\ol{\beta}:  {\wt R\san2 H}\to { \wt S\san2 G}.
	$$
\end{lemma}
\begin{proof}
It is already known that $\beta$ gives rise to a holomorphic map $\wt\beta:R\to \wt S\san2 G$. Points with isomorphic fibers are mapped to points with isomorphic fibers so that the topological space $R/H$ is mapped homeomorphically onto an open subset of the topological space $S/G$, and points corresponding to polystable bundles i.e.\ polystable points with respect to the group action of $H$ go to polystable points with respect to $G$. In particular the holomorphic map $\wt \beta$ is constant on $G$-orbits. Any holomorphic function on $\wt S\san2 G$ is pulled back under $\ol \beta$ to an $H$-invariant function on $\wt R$. The claim follows again from Theorem~\ref{th:invfct}.
\end{proof}

In order to descend the map $\alpha: (S,s_0) \to (R,0) $ to GIT-quotients, first a slice $Q\subset S$ through $s_0$ for the action of $G$ on $S$ is chosen (and $\alpha$ is restricted to $Q$) so that $G\times Q\san2 G_{s_0} \subset \wt S \san2 G$ is an open neighborhood of the image of $s_0$.

\begin{lemma}\label{le:alphadesc}
	The map $\alpha: (S,s_0) \to (R,0) $ descends to a holomorphic map
	$$
	{\wt S\san2 G} \supset {\wt Q\san2 G_{s_0}} \stackrel{\ol{\alpha}}{\longrightarrow} {\wt R\san2H}
	$$
	where $\wt Q\san2 G_{s_0}\subset \wt S\san2 G$ is an open neighborhood of the image of $s_0$.
\end{lemma}
\begin{proof}
	Once the analytic slice is chosen, the proof is the same as the proof of Lemma~\ref{le:betadesc}.
\end{proof}
Note that the analytic slices are not shown to be parameter spaces of semi-universal deformations -- identification takes place only after passing to the analytic GIT-quotient.

\begin{remark}
  By \cite[Theorem 4.7]{bu-sch} it follows that $H=G_{s_0}$.
\end{remark}

\begin{proposition}\label{pr:alpbet}
Let $\xi$ be a semi-universal deformation with parameter space $(S,0)$, let $s_0\in S$ be a polystable point, denote by $\xi_{s_0}$ the induced deformation of the fiber at $s_0$, and let $\zeta$ be a semi-universal deformation of $\xi_{s_0}$ over a base space $(R,0)$. Let  $\alpha: (S,s_0) \to (R,0)$ and $\beta : (R,0) \to (S,s_0)$ be holomorphic maps of induced space germs such that $\alpha^* \zeta = \xi_{s_0}$ and $\beta^*\xi_{s_0} = \zeta$.  Let $Q$ be a slice for the action of $G$ on $S$ through $s_0$ so that $Q\san2 G_{s_0}$ can be identified with a neighborhood of the image of $s_0$ in $S\san2 G$.

Let $G$ and $H$ be the automorphism groups of the central fiber and  of the fiber at $s_0$ resp. Then the base change morphisms descend to isomorphisms of analytic GIT-quotients near the points $0$ and $s_0$ after replacing the respective spaces by neighborhoods of the distinguished points:
$$
\xymatrix{
{R\san2 H} \ar[r]^{\ol{\beta}}_\sim & {Q\san2 G_{s_0}} \ar[r]^{\ol{\alpha}}_\sim & {R\san2 H}
}
$$
\end{proposition}
\begin{proof}
The above holomorphic maps were constructed in Lemma~\ref{le:betadesc} and Lemma~\ref{le:alphadesc}. The underlying topological spaces consist of polystable points so that $\ol \alpha$ and $\ol\beta$ are homeomorphisms, and on the complements of  thin analytic sets these maps are biholomorphic. Moreover, by Lemma~\ref{le:gitiso1} the map $\ol\alpha \circ \ol\beta$ is an isomorphism on the whole space. This is only possible, if both maps are isomorphisms  (cf.\cite[Chap.~8 Sect.\ 2]{g-r} on one-sheeted coverings).
\end{proof}
\begin{remark}\label{re:stable}
  If $\cE_{s_0}$ is stable, then $H\simeq \C^*$ consists of homotheties acting trivially.
\end{remark}

\section{Application to moduli of holomorphic vector bundles}\label{se:appmodvb}
The aim is to construct a complex space whose local models are analytic GIT-quotients such that the complement of a closed analytic set is the coarse moduli space of stable holomorphic vector bundles.

The local model was constructed in Theorem~\ref{th:localmain}. By general theory, in particular the Kempf-Ness Theorem, and Proposition~\ref{pr:trfree}, and by Proposition~\ref{pr:Cor46} the points of the space correspond to isomorphism classes of polystable vector bundles. An analytic version of Luna's slice theorem was shown and Proposition~\ref{pr:alpbet} provides the gluing of local analytic GIT models. Namely, when two such models contain a point with (polystable) isomorphic fibers $E$ say, each of the models is locally isomorphic to the GIT-quotient constructed from a semi-universal deformation of $E$.

Finally the Hausdorff property of the resulting space precisely follows from the same argument as the proof of Lemma~4.2.4 in \cite{d-k} (see also \cite[\S 6]{AHS}).

The space $\cM_{GIT}$ contains the coarse moduli space of stable holomorphic vector bundles by Remark~\ref{re:stable} and Remark~\ref{re:pslocus}.

\begin{theorem}\label{th:main}
Given a compact \ka\ manifold there exists a complex space $\mathcal{M}_{GIT}$, whose local models are analytic GIT-quotients of parameter spaces of holomorphic vector bundles with a restricted group action.

The space $\cM_{GIT}$ contains the coarse moduli space $\mathcal M$ of stable holomorphic vector bundles as an open subspace. The complement $\cM_{GIT}\backslash \cM$ is a closed analytic subspace, whose points correspond to isomorphism classes of polystable, non-stable vector bundles.
\end{theorem}

\begin{definition}
  Let $\mathcal K$ be the class of polystable (including stable) vector bundles on the given compact \ka\ manifold. A reduced complex space $\cN$ is called a {\em classifying space} for $\mathcal K$, if the following conditions hold.
  \begin{itemize}
    \item[(i)] The points of $\cN$ correspond to isomorphism classes of polystable, holomorphic vector bundles on the \ka\ manifold $X$,
    \item[(ii)] Let $\cE$ be a holomorphic family of holomorphic vector bundles on $X\times Z$, where $Z$ is a reduced complex space, such that the fiber $\cE_{z_0}=\cE|X\times\{z_0\}$, of a point $z_0\in Z$ is polystable. Then, after replacing $Z$ by a neighborhood of $z_0$ if necessary, there exists a unique holomorphic map $\varphi : Z \to \cN$ such that $\varphi(z)\in \cN$ corresponds to the isomorphism class of $\cE_z$ provided $\cE_z$ is polystable.
    \item[(iii)] If $\wt \cN$ is a further reduced complex space satisfying {\rm (i)} and {\rm (ii)} with unique holomorphic maps $\wt \varphi: Z\to \wt\cN$ for families over spaces $Z$ in the sense of {\rm (ii)}, then there exists a unique holomorphic map $\chi: \cN \to \wt \cN$ such that $\chi\circ \varphi = \wt \varphi$.
  \end{itemize}
\end{definition}
The corresponding condition (ii') for a {\em coarse moduli space} $\cN_{\text{coarse}}$ reads
\begin{itemize}
  \item[(ii')] Let $\cE$ be a holomorphic family of polystable holomorphic vector bundles on $X\times Z$, where $Z$ is a reduced complex space. Then there exists a unique holomorphic map $\varphi : Z \to \cN_{\text{coarse}}$ such that $\varphi(z)\in \cN_{\text{coarse}}$ corresponds to the isomorphism class of $\cE_z$.
\end{itemize}
Condition (iii) for a  coarse moduli space would refer to a space satisfying (i) and (ii'). A classifying space for $\cK$ certainly satisfies (ii'). However, since in general there exists no deformation theory for the class $\cK$, it would have to be shown that there exist sufficiently many holomorphic families with fibers from $\cK$ so that condition (iii) still holds.

\begin{theorem}
  The space $\cM_{GIT}$ is a classifying space for the class of polystable holomorphic vector bundles on a compact \ka\ manifold.
\end{theorem}

\begin{proof}
  The first condition follows from the construction, which identifies local models, if the polystable fibers are isomorphic.

  Let a holomorphic family $\zeta$ in the sense of (ii) be given. After replacing $Z$ by a neighborhood of $z_0$ the given family is isomorphic to the pull-back $\psi^*\xi$ of a semi-universal family $\xi$ over a space $(S,s_0)$ where $\psi:(Z,z_0) \to (S,s_0)$ is the base change map. Let $G$ be the automorphism group of the central fiber. The analytic quotient $p:S\to S\san2 G\subset \cM_{GIT}$ is a local model, and we set $\varphi= p\circ \psi$. So far $\psi$ is uniquely determined by the deformation $\zeta$ only on the set of polystable points. If $z\in Z$ is any point, and $s=\psi(z)$, then the image in $S\san2 G$ corresponds to a polystable point in the closure $\ol{G\cdot s}$ of the orbit of $s$ which is unique up to the action of $K$. Let $\psi$ and $\wt \psi$ be two choices for the same deformation $\zeta$. Then for any $z\in Z$, and $s=\psi(z)$, $\wt s=\wt\psi(z)$ the bundles $\cE_s$ and $\cE_{\wt s}$ are isomorphic. By Proposition~\ref{pr:Cor46} the points $s$ and $\wt s$ are in the same $G$-orbit, hence the images in $S\san2 G$ are the same. This shows existence and uniqueness of $\varphi$.

  In order to prove (iii) a map $\chi : \cM_{GIT} \to \wt\cN$ has to be constructed. The pointwise definition and uniqueness of any such map  $\chi$ follows because of condition (i) holding for both $\cM_{GIT}$ and $\wt \cN$. The holomorphicity of $\chi$ is shown as follows: Let $(S,0) \stackrel{p}{\longrightarrow}S\san2 G \subset \cM_{GIT}$ be a local model. Then by (ii) for $\wt\cN$ there exists a holomorphic function $\wt\varphi : S \to \wt\cN$ so that $\chi\circ p = \wt\varphi$ pointwise. The map can be extended to the respective $G$-invariant space $\wt S$.
  $$
  \xymatrix{
  \wt S \ar[r]^p \ar[drr]_{\wt\varphi} & \wt S\san2 G\,\,\,  \subset \hspace{-8mm} &\cM_{GIT}\ar[d]^\chi\\ && \wt\cN
  }
  $$
  Let $Q$ be an open neighborhood of $\wt\varphi(0)$. Then holomorphic functions on $Q$, pulled back to $\wt S$ are $G$-invariant, hence holomorphic on $\chi^{-1}(Q) \subset \cM_{GIT}$ by Theorem~\ref{th:invfct}. This shows that $\chi$ is holomorphic. For families over arbitrary spaces $Z$ in the sense of condition (ii) the construction is compatible with a base change map $\psi:Z\to S$, where $S$ is the base of a semi-universal deformation. This implies the claim together with the uniqueness of maps from parameter spaces to $\cM_{GIT}$ and $\wt\cN$ resp.
\end{proof}

The question about the existence of a coarse moduli space of polystable vector bundles will be treated by the authors in a forthcoming article.

\section{Generalized \wp\ metric}\label{se:WP}

\subsection{General theory}\label{se:genth}
The coarse moduli space of stable holomorphic vector bundles on a compact \ka\ manifold is known to carry a \ka\ form $\omega^{WP}$. For notation and properties we refer to \cite{b-s}, and only give some essential facts. As a Hermitian metric it is defined as the $L^2$-inner product of  \ks\ tensors that are harmonic with respect to \he\ metrics on the given bundle and the \ka\ form on the fixed manifold.

Let $(X,\omega_X)$ be a compact \ka\ manifold. Then a holomorphic family $\{\cE_s\}_{s\in S}$ of stable vector bundles parameterized by a reduced complex space $S$ is given by a holomorphic vector bundle $\cE$ on $X\times S$, such that $\cE_s=\cE|X\times \{s\}$. Set $E=\cE_{s_0}$ for $s_0\in S$. Let $h$ be a Hermitian metric such that $h|\cE_s$ is \he, and denote by $F$ the curvature form of $h$, where
$$
\Lambda(\ii F|\cE_s)= \lambda_E \cdot id_{\cE_{s}}\qquad \lambda_E= \frac{2\pi (c_1(\cE_s)\we \omega_X^{n-1})[X]}{rank(\cE_s)\cdot vol(X)}.
$$
Let again
$$
\rho: T_{s_0}S \to H^1(X, End(E))
$$
be the \ks\ map and $v\in  T_{s_0}S$ a tangent vector.
Then the \he\ condition implies immediately that the harmonic representative $A_v$ of $\rho(v)$ equals the contraction of the tangent vector with the curvature form $F$ of $(\cE,h)$ resulting in a $\ol\pt$-closed $(0,1)$-form of class $\cinf$:
\begin{equation}\label{eq:Av}
A_v= v\, \lrcorner\, \ii F|E \in \cA^{0,1}(X,End(E)).
\end{equation}
Namely:
$$
\ol\pt^*(v\, \lrcorner\, \ii F|E)= v(\Lambda(\ii F|\cE_{s_0}))=v(\lambda_E\cdot id _{Es_0})=0,
$$
where the tangent vector $v\in T_{s_0}S$ is lifted to a vector field along $X\times\{s_0\}$  in $S$-direction.

This equation also holds for polystable structures, so it is meaningful to extend the \wp\ inner product to polystable points.

The above fact yields that the \wp\ form satisfies a general fiber integral formula, which immediately implies that the result is a positive definite, $d$-closed, real $(1,1)$-form:
\begin{equation}\label{eq:fbint}
\omega^{WP}_S= \frac{1}{2}\int_{X\times S/S} \tr(F\we F) \we \omega^{n-1}_X + \lambda_E\int_{X\times S/S} tr(\ii F) \we \omega^{n}_X.
\end{equation}
The second term in \eqref{eq:fbint} and an unwanted contribution in the first term cancel out.
By rescaling the hermitian metric on $\cE$ both of these can assumed to be equal to zero.

\subsection{Application to the classifying space $\cM_{GIT}$} Let $E$ be a polystable vector bundle with a semi-universal deformation parameterized by a reduced complex space $S\subset U\subset V$ such that $E=\cE_{s_0}$ in the situation of Sections \ref{se:appdefvb} and \ref{se:appmodvb}. The local model for the classifying space is equal to $S\san2 G$, and its underlying topological space is equal to the quotient $(\mu^{-1}(0)\cap S)/ K$.

According to Section~\ref{se:appdefvb}, in particular Remark~\ref{re:semi-c} a semi-connection
on $E\times U$ over $X\times U$ exists that is integrable over $X\times S$. Denote by $F$ the curvature form. Then the righthand side of the fiber integral $\eqref{eq:fbint}$ taken over $X\times U/U$ defines a symplectic form $\omega_U$ of class $\cinf$.
\begin{proposition}\label{pr:2form}
The restriction of the form $\omega_U$ to $S$ is of type $(1,1)$. The form $\omega_U|(\mu^{-1}(0)\cap U)$ is $K$-invariant. Its restriction to the stable locus $S'\subset S\subset U$ is the pull-back of the \wp\ form. At polystable points the form $\omega^{WP}_S$ defines the inner product of harmonic \ks\ tensors. At $0\in S$ it is positive definite.
\end{proposition}
\begin{proof}
  At integrable points the curvature $F$ is of type $(1,1)$ so that $\omega_U$ is of type $(1,1)$ on $S$. At points of $\mu^{-1}(0)$ the relevant connection is $K$-invariant, so is $F$, which shows that the fiber integral has this property. At points of $S\cap \mu^{-1}(0)$ equation \eqref{eq:Av} is also valid so that \eqref{eq:fbint} defines the \wp\ inner product.
\end{proof}
For the restriction of the moduli space $\cM_{GIT}$ to a (smooth) connected component $\cM_0$ consisting of objects whose semi-universal deformations are unobstructed more can be said. Singular symplectic quotients and locally stratified spaces were studied by  Sjamaar and Lerman \cite{s-l}. Here the complex structure of the classifying space is of interest. We assume that on $\cM_0$ the stable points form an everywhere dense subset.
\begin{theorem}
The \wp\ form on the stable locus of $\cM_0$ extends as a positive $(1,1)$-current $\omega^{WP}_0$  to $\cM_0$ that possesses locally a continuous $\pt\ol\pt$-potential $\chi$.
\end{theorem}
\begin{proof}
  The classifying space is locally of the form $\wt U\san2 G$ with $S=U$ in the above notation. The space $U$ can be chosen in a way that $\omega_U$ possesses a $\pt\ol\pt$-potential $\chi$ of class $\cinf$. A $\cinf$ $K$-invariant function  on $\mu^{-1}(0)$ is given by integration over the group $K$ which descends to a continuous function $\ol\chi$ on the quotient   $\mu^{-1}(0)/K$. At the open locus of stable points, locally the action of $K$ is of maximal differentiable rank so that $\ol\chi$ is of class $\cinf$ on the set of stable points and is a potential for the \wp\ form there. Because of the continuity of local potentials these define a global extension as a positive current (with positivity being primarily defined on the normalization).
\end{proof}
\subsection{The \wp\ form for projective manifolds}
The aim is to define an extension of the \wp\ form also if the parameter spaces for polystable bundles are singular. This is possible for projective manifolds $X$ such that $\omega_X=c_1(\cL,h_\cL)$ for a positive holomorphic line bundle $\cL$. The following general fact about determinant line bundles and Quillen metrics is applicable (cf.\ e.g.\ \cite{b-s}, in particular $(4.11)$ for the definition of the determinant line bundle $\lambda$ equipped with the Quillen metric $h^Q$ on the base $S$ for a family of hermitian holomorphic vector bundles on $X\times S$, and $(4.13)$ for the Chern form of $(\lambda,h^Q)$). The formulas hold for a holomorphic family $(\cE, h)$ of hermitian vector bundles of rank $r$ over $X\times S$. The determinant line bundle $\lambda$ on $S$ is defined in terms of a certain virtual holomorphic vector bundle:
\begin{equation}\label{eq:detbdl}
  \lambda= \det \ul{\ul R}p_* \left( (End(\cE)-\cO^{r^2}_{X\times S} ) \otimes( p^*\cL -p^*\cL^{-1})^{\otimes(N-1)}  \right)
\end{equation}
The Quillen metric $h^Q$ depends upon the fiberwise $L^2$ norms and the zeta function with respect to the fiberwise Dirac operators/Laplacians of the holomorphic bundles that are involved. In particular the Quillen metric is of class $\cinf$. The formula of Bismut, Gillet, and Soul\'e (cf.\ \cite{b-s} for a list of references) implies \cite[(4.13)]{b-s} so that
\begin{equation}\label{eq:bgsrr}
   c_1(\lambda,h^Q)= - \left(\int_{X\times S/S} {\rm td}(X\times S/S,\omega_X)\cdot {\rm ch}(\cF,h) \right)^{(1,1)},
\end{equation}
where $\mathrm{td}(X\times S/S)$ denotes the relative Todd character form, and where $\cF$ denotes the virtual bundle on the righthand side of \eqref{eq:detbdl}. For singular spaces $S$ see \cite{f-s}. The fiber integral formula \eqref{eq:fbint} for the \wp\ form together with this formula imply the following fact.
\begin{fact}
Let a semi-universal deformation of a polystable holomorphic vector bundle be given with parameter space $S$. Then on the stable locus $S'$ up to a fixed positive numerical constant the curvature form of the determinant line bundle $(\lambda,h^Q)$ is equal to the pull-back of the \wp\ form $\omega^{WP}_{S'}$:
$$
c_1(\lambda,h^Q)|S'\simeq \omega^{WP}_{S'}.
$$
\end{fact}
In general the fiber integral \eqref{eq:fbint} only implies that $\omega^{WP}_S$ is the restriction of a symplectic form on the smooth ambient space $U$ to the reduced analytic subspace $S$. In the case of a Hodge manifold $X$ it possesses a $\pt\ol\pt$-potential of class $\cinf$. Now the arguments of Section~\ref{se:genth} are applicable to the situation where semi-universal deformations are not necessarily unobstructed.
\begin{theorem}
  Let $(X,\omega_X)$ be a Hodge manifold with $\omega_X=c_1(\cL,h_\cL)$, and denote by $\cM_{GIT_0}$ the moduli space of polystable vector bundles that can be deformed locally into stable vector bundles. Then the \wp\ form on the stable locus extends to $\cM_{GIT_0}$ as a positive current with local continuous $\pt\ol\pt$-potentials.
\end{theorem}
The abstract methods of Dervan and Naumann \cite{d-n} are applicable now. We indicate their argument very briefly. The coherent sheaf on the GIT-quotient for a local model defined by $G$-invariant sections of the respective determinant line bundle becomes invertible, when the determinant line bundle is replaced by a finite power  given by the number of connected components of $G$ (cf.\ \cite{sj}). Using Kirwan's stratification of GIT quotients \cite{Kir} they show that the continuous potentials for the \wp\ metric give rise to continuous (singular) hermitian metrics on these line bundles, which is of class $\cinf$ on the stable locus using the stratification of the GIT-quotients. Altogether the results of Dervan and Naumann yield the following.

\medskip

\begin{appl*}
Derived from determinant line bundles equipped with Quillen metrics there exists a holomorphic line bundle $\lambda_{\cM_0}$ on $\cM_0$, equipped with a continuous (singular) hermitian metric, which is of class $\cinf$ on the stable locus. Its curvature current is positive, and equal to the \wp\ form on the stable locus.
\end{appl*}

\end{document}